\documentclass[oneside,reqno]{amsart} 

\newcounter{version}
\newcounter{short}
\newcounter{long}
\usepackage{amsmath,amsthm,amscd}
\usepackage{amssymb,amsfonts} 
\input xy
\xyoption{all} \numberwithin{equation}{section}

\theoremstyle{definition}
\newtheorem{dfn}{Definition}[section]

\newtheorem{rem}[dfn]{Remark}
\theoremstyle{plain}
\newtheorem{thm}[dfn]{Theorem}
\newtheorem{prp}[dfn]{Proposition}
\newtheorem{cor}[dfn]{Corollary}
\newtheorem{lem}[dfn]{Lemma}

\newcommand\C{{\mathbb C}}
\newcommand\Z{{\mathbb Z}}
\newcommand\N{{\mathbb N}}

\newcommand\T{{\mathbb T}}

\newcommand\LL{{\mathcal L}}

\newcommand\Id{{\rm Id}}
\newcommand\ad{{\rm ad}}
\newcommand\Ann{{\rm Ann}}
\newcommand\Hom{{\rm Hom}}
\newcommand\End{{\rm End}}
\newcommand\Span{{\rm Span}}
\newcommand\Der{{\rm Der\,}}

\newcommand\darrow{\longrightarrow {\mkern -27mu} {\raise 6pt \hbox{\it d}} {\mkern 16mu}}

\newcommand\del{\partial}

\date{}
\begin{document}

\

\title
[Simple cuspidal modules for solenoidal Lie algebras]
{Classification of simple cuspidal modules for solenoidal Lie algebras}
\author{Yuly Billig}
\address{School of Mathematics and Statistics, Carleton University, Ottawa, Canada}
\email{billig@math.carleton.ca}
\author{Vyacheslav Futorny}
\address{ Instituto de Matem\'atica e Estat\'\i stica,
Universidade de S\~ao Paulo,  S\~ao Paulo,
 Brasil}
 \email{futorny@ime.usp.br}
\let\thefootnote\relax

\begin{abstract}
We give a new conceptual proof of the classification of cuspidal modules for the solenoidal Lie algebra.
This classification was originally published by Y.~Su \cite{Su}. Our proof is based on the theory of
modules for the solenoidal Lie algebras that admit a compatible action of the commutative algebra 
of functions on a torus.
\end{abstract}

\maketitle

\section{Introduction.}

In \cite{BF} we classify simple $W_n$-modules with finite-dimensional weight spaces. This generalizes the 
classical result of Mathieu \cite{Mat} on simple modules with finite-dimensional weight spaces for the Virasoro
algebra. Mathieu proved that simple weight modules fall into two classes: (1) highest/lowest weight modules and
(2) modules of tensor fields on a circle and their quotients.

In \cite{BF} we use solenoidal Lie algebras as a bridge between the Lie algebra $W_1$ of vector fields on a circle, 
and the Lie algebra $W_n$ of vector fields on an $n$-dimensional torus. A solenoidal Lie algebra $W_\mu$ is a subalgebra
in $W_n$ which consists of vector fields collinear to a generic vector $\mu$ at each point of the torus. Integral curves of
such vector fields are dense windings of a torus, which motivates the term ``solenoidal''. Solenoidal Lie algebras are
also known in the literature as the centerless higher rank Virasoro algebras.

In many ways solenoidal Lie algebras behave like $W_1$, yet they are graded by $\Z^n$, like $W_n$. In fact, $W_n$
can be decomposed (as a vector space) into a direct sum of $n$ solenoidal subalgebras.

The main result of \cite{BF} relies on the classification of simple weight modules for the solenoidal Lie algebras with 
a property that the dimensions of all weight spaces are uniformly bounded. We call such modules cuspidal.

Classification of cuspidal modules for the solenoidal Lie algebras was published by Su \cite{Su}. However
the methods of \cite{Su} are extremely computational and we were unable to verify all the details. In the present
paper we offer an alternative conceptual proof of this classification. Our proof is based on the technique developed in \cite{Jet}
and \cite{BF}.

 Using the classification of simple cuspidal modules, Lu and Zhao \cite{LZ} classified all simple modules for the solenoidal
Lie algebras with finite-dimensional weight spaces.

It follows from the result of Mathieu \cite{Mat} that the classification of simple cuspidal $W_1$-modules is given by
the modules of tensor fields on a circle $T(\alpha, \beta)$, $\alpha,\beta\in\C$. The modules $T(\alpha,\beta)$
have bases $\left\{ v_s \, | \, s \in \beta + \Z \right\}$ and the action of $W_1$ is given by
\begin{equation}
e_k v_s = (s + \alpha k) v_{s+k}, \quad k \in\Z, \; s \in\beta+\Z .
\label{actT}
\end{equation} 
Here $\left\{ e_k \, | \, k\in\Z \right\}$ is a basis of $W_1$ and its Lie bracket is
\begin{equation}
[e_k, e_m] = (m-k) e_{m+k}, \quad k,m\in\Z .
\label{Witt}
\end{equation}
A simple cuspidal $W_1$-module is isomorphic to one of the following:

$\bullet$ a module of tensor fields $T(\alpha, \beta)$ with $\alpha \neq 0, 1$ or $\alpha = 0, 1$ and
$\beta \not\in\Z$,

$\bullet$ the quotient $\overline{T}(0,0)$ of the module $T(0,0)$  of functions by a one-dimensional 
submodule of constant functions,

or

$\bullet$ the trivial 1-dimensional module.

For a generic vector $\mu\in\C^n$ we consider a lattice $\Gamma_\mu \subset \C$, where $\Gamma_\mu$
is the image of $\Z^n$ under the map $r\in\Z^n \mapsto \mu \cdot r \in \C$. Then the solenoidal Lie algebra
$W_\mu$ may be presented as the Lie algebra with a basis $\left\{ e_k \, | \, k \in \Gamma_\mu \right\}$ and the
Lie bracket \eqref{Witt} with $k, m\in\Gamma_\mu$ instead of $\Z$. Likewise, the $W_\mu$-module $T(\alpha, \beta)$,
$\alpha, \beta\in\C$ has a basis $\left\{ v_s \, | \, s \in \beta + \Gamma_\mu \right\}$ and the action of $W_\mu$
is still given by the formula \eqref{actT} but with a change $k \in \Gamma_\mu$, $s\in \beta+\Gamma_\mu$.

The classification of simple cuspidal $W_\mu$-modules is the same as in the case of $W_1$, with a replacement of
the condition $\beta\not\in \Z$ for the modules $T(\alpha,\beta)$, $\alpha = 0, 1$, with the condition
$\beta\not\in\Gamma_\mu$.

A special class of modules plays a crucial role in our approach -- these are $W_\mu$-modules that admit a compatible action of 
the commutative algebra $A$ of functions on a torus. One of our key results is that every simple cuspidal $W_\mu$-module
is a homomorphic image of a simple $AW_\mu$-module. Using the methods developed in \cite{Jet} we show that
simple cuspidal $AW_\mu$-modules are precisely the modules of tensor fields $T(\alpha, \beta)$. These two results combined 
yield the  classification of simple cuspidal $W_\mu$-modules.


The present paper is intertwined with our recent work \cite{BF}. We use in \cite{BF} the classification of cuspidal
$W_\mu$-modules. In the present work we use some of the methods and techniques developed in \cite{BF},
which are independent of the classification of $W_\mu$-modules.

The paper is organized as follows. In Section 2 we review solenoidal Lie algebras, introduce the family $T(\alpha, \beta)$
of $W_\mu$-modules and exhibit their elementary properties. In Section 3 we study the structure of cuspidal
$AW_\mu$-modules and prove that every simple cuspidal $AW_\mu$-module is isomorphic to $T(\alpha, \beta)$ for
some $\alpha, \beta\in\C$. In Section 4 for each cuspidal $W_\mu$-module we construct its $A$-cover and use this 
construction to establish our main result.

\section{Solenoidal Lie algebras and their cuspidal modules}

 Consider the Lie algebra $W_n$ of vector fields on an $n$-dimensional torus. The algebra of 
(complex-valued) Fourier polynomials on $\T^n$ is isomorphic to the algebra of Laurent polynomials
$$A = \C [t_1^{\pm 1}, \ldots, t_n^{\pm 1}] ,$$ 
and $W_n$ is the Lie algebra of derivations of $A$. Thus $W_n$ has a natural structure of an $A$-module,
which is free of rank $n$. We choose $d_1 = t_1 \frac{\partial}{\partial t_1}, \ldots, d_n = t_n \frac{\partial}{\partial t_n}$
as a basis of this $A$-module:
$$W_n = \mathop\oplus\limits_{p=1}^n A d_p .$$
In case when $n=1$ we get the Witt algebra $W_1$ with basis $e_k = t_1^k d_1$, $k\in\Z$, and  the bracket
\begin{equation}
[e_k, e_m ] = (m-k) e_{m+k} .
\label{witt}
\end{equation}
This paper is devoted to one particular family of subalgebras in $W_n$.
\begin{dfn}
We call a vector $\mu\in\C^n$ {\it generic} if  $\mu \cdot r \neq 0$ for all $r \in\Z^n\backslash \{ 0 \}$.
\end{dfn}
\begin{dfn}
Let $\mu$ be a generic vector in $\C^n$ 
and let $d_\mu = \mu_1 d_1 + \ldots + \mu_n d_n$. 
A {\it solenoidal} Lie algebra $W_\mu$ is the subalgebra in $W_n$ which consists of vector fields
collinear to $\mu$ at each point of $\T^n$, $W_\mu = A d_\mu$.
\end{dfn}
The Lie bracket in $W_\mu$ is 
\begin{equation}
[t^r d_\mu, t^s d_\mu ] = \mu \cdot (s-r) t^{r+s} d_\mu ,
\label{sol}
\end{equation}
with $r, s \in \Z^n$. Here we are using the notation  $t^r = t_1^{r_1} \ldots t_n^{r_n}$ for
$r = (r_1,\ldots,r_n) \in \Z^n$.

Denote by $\Gamma_\mu$ the image of $\Z^n$ under the embedding $\Z^n \rightarrow \C$
given by $r \mapsto \mu \cdot r$. Then we can view the solenoidal Lie algebra $W_\mu$ as a
version of the Witt algebra $W_1$ where the indices of the basis elements $e_r$ run not over $\Z$,
but over the lattice $\Gamma_\mu \subset \C$. Here we make the identification $t^r d_\mu = e_{\mu \cdot r}$
and the formula \eqref{witt} remains valid.

 Let us now discuss modules for the solenoidal Lie algebras. A $W_\mu$-module $M$ is called a {\it weight} module
if $M = \mathop\oplus\limits_{\lambda\in\C} M_\lambda$, where the weight space $M_\lambda$ is defined as
$$M_\lambda = \left\{ v \in M \, | \, d_\mu v = \lambda v  \right\} .$$

In particular, $A$ is a weight $W_\mu$-module and $W_\mu$ is a weight module over itself.

Any weight $W_\mu$-module can be decomposed into a direct sum of submodules corresponding to distinct cosets of $\Gamma_\mu$
in $\C$. It is then sufficient to study modules supported on a single coset $\beta + \Gamma_\mu$, $\beta\in \C$.
We are going to impose such a condition on $M$ and fix $\beta$ for the rest of this paper.

\begin{dfn}
A weight module is called {\it cuspidal} if the dimensions of its weight spaces are uniformly bounded by some
constant.
\end{dfn}

 Let us construct a family of cuspidal $W_\mu$-modules. Fix $\alpha, \beta \in \C$.
\begin{dfn}
A module of {\it tensor fields} $T(\alpha, \beta)$ is a vector space with a basis $\{ v_s \, | \, s \in \beta + \Gamma_\mu \}$ and the following action
of $W_\mu$:
\begin{equation}
e_k v_s = (s + \alpha k) v_{s+k} , \quad k \in \Gamma_\mu, \; s \in \beta + \Gamma_\mu.
\label{tenden}
\end{equation}
\end{dfn}

 It is easy to check that $T(\alpha, \beta)$ is indeed a $W_\mu$-module. These modules are cuspidal since every weight space in
$T(\alpha, \beta)$ is 1-dimensional.

\begin{prp}
(1) The $W_\mu$-module $T(\alpha, \beta)$ is simple unless $\alpha = 0, 1$ and $\beta + \Gamma_\mu = \Gamma_\mu$. 

(2) The $W_\mu$-module $T(0,0)$ has a trivial 1-dimensional submodule $\C v_0$,
and the quotient  $\overline{T}(0,0) = T(0,0) / \C v_0$ is a simple $W_\mu$-module.

(3) Let $\{ v_s | s \in\beta+\Gamma_\mu \}$ be a basis of $T(0,\beta)$ and  $\{ v^\prime_s | s \in\beta+\Gamma_\mu \}$ be a basis of $T(1,\beta)$.
The map $\theta: T(0,\beta) \rightarrow T(1,\beta)$,
$$\theta( v_s) = s v^\prime_s$$
is a homomorphism of $W_\mu$-modules.

(4)  $\theta(T(0,0)) \cong \overline{T}(0,0)$ is a submodule of codimension $1$ in $T(1,0)$.
\label{red}
\end{prp}
\begin{proof}
Clearly, every submodule in $T(\alpha,\beta)$ is homogeneous with respect to the weight decomposition.
Thus $T(\alpha, \beta)$ is simple if and only if for every $s \in \beta+\Gamma_\mu$ the vector $v_s$
generates $T(\alpha, \beta)$. Let us asume $\alpha \neq 0$. Fix $s \in \beta+\Gamma_\mu$. Since
$$ e_k v_s = (s + \alpha k) v_{s+k}, $$
and $s + \alpha k$ may vanish only for $k = - s/\alpha$, we see that a $W_\mu$-submodule generated by $v_s$ is
either $T(\alpha,\beta)$ or has codimension 1 in $T(\alpha, \beta)$. The latter happens precisely when
for some $m\in \beta+\Gamma_\mu$
$$e_k v_{m-k} = (m-k+\alpha k) v_m = 0$$
for all non-zero $k \in \Gamma_\mu$. It follows from this that $T(\alpha, \beta)$ with $\alpha \neq 0$ has a submodule of codimension $1$
only when $\alpha =1$ and $m = 0 \in \beta + \Gamma_\mu$.

Consider now the case $\alpha = 0$. Then $e_k v_s = s v_{s+k}$, and every $v_s$ with $s \neq 0$ generates $T(0, \beta)$,
thus $T(0, \beta)$ could only have a 1-dimensional submodule $\C v_0$. This shows that $T(0, \beta)$ is reducible
if and only if $0 \in \beta + \Gamma_\mu$.
This establishes parts (1) and (2) of the proposition. Part (3) can be verified by a straightforward direct calculation. Claim (4) is then an immediate consequence of (3). 
\end{proof}

\section{$AW_\mu$-modules}
In this section we are going to describe the structure of cuspidal $W_\mu$-modules that admit a compatible action of the commutative 
algebra $A$ of functions on a torus.
\begin{dfn}
We call $M$ an $AW_\mu$-module if it is a module for both the solenoidal Lie algebra $W_\mu$ and the commutative unital algebra
$A = \C [t_1^{\pm 1}, \ldots, t_n^{\pm 1}]$ with these two structures being compatible:
\begin{equation}
x (f v) = (x f) v + f (x v), \quad f \in A, \; x \in W_\mu, \; v \in M.
\label{compat}
\end{equation}
\end{dfn} 

If an $AW_\mu$-module $M$ is a weight module then \eqref{compat} implies that the action of $A$ is compatible 
with the weight grading of $M$: $A_\gamma M_\lambda \subset M_{\lambda+\gamma}$, $\gamma \in \Gamma_\mu$,
$\lambda \in \beta + \Gamma_\mu$. Suppose an $AW_\mu$-module $M$ has a weight decomposition with one of the
weight spaces being finite-dimensional. Since all non-zero homogeneous elements of $A$ are invertible, we 
conclude that all weight spaces of $M$ have the same dimension and that $M$ is a free $A$-module of a finite
rank.

Let us now assume that $M$ is a cuspidal $AW_\mu$-module. Let $U = M_\beta$, $\dim U < \infty$. 
Since $M$ is a free $A$-module, it can be presented as $M \cong A \otimes U$.

For each $s \in \Z^n$ consider an operator
$$ D(s) : \, U \rightarrow U,$$
given by $D(s) = t^{-s} \circ (t^s d_\mu)$.
Note that $D(0) = \beta \Id$.

The finite-dimensional operator $D(s)$ completely determines the action of $t^s d_\mu$ on $M$ since
by \eqref{compat} 
\begin{equation}
(t^s d_\mu) (t^m u) = (t^s d_\mu t^m) u + t^m (t^s d_\mu u) = \mu \cdot m t^{m+s} u + t^{m+s} D(s) u .
\label{tensact}
\end{equation}
From \eqref{tensact} and the commutator relations \eqref{sol} we can easily derive the Lie bracket (cf., Lemma 1 in \cite{Jet}):
\begin{equation}
[D(s), D(m)] = \mu \cdot m \left( D(m+s) - D(m) \right) -  \mu \cdot s \left( D(m+s) - D(s) \right).
\label{deform}
\end{equation}
We are going to show that $D(s)$ can be expressed as a polynomial in $s = (s_1, \ldots, s_n)$.

Let us state two standard results on polynomial and rational functions on $\Z^n$. The proofs are elementary and we omit them. 
\begin{lem}
Let $P_1$, $P_2$ be two polynomials in $\C [x_1, \ldots, x_n]$ such that the degree of $P_1$, $P_2$ in each variable does not
exceed $K\in\N$. Let $B = B_1 \times \ldots \times B_n \subset \Z^n$ with $|B_i| > K$. If $P_1(x) = P_2(x)$ for all
$x \in B$ then $P_1 = P_2$ in  $\C [x_1, \ldots, x_n]$.
\end{lem}
\begin{cor}
Let $R_1 = {P_1}/{Q_1}$,    $R_2 = {P_2}/{Q_2}$ with $P_1, P_2, Q_1, Q_2 \in \C [x_1, \ldots, x_n]$
such that the degree of $P_1, P_2, Q_1, Q_2$ in each variable does not exceed $K$. 
Let $B = B_1 \times \ldots \times B_n \subset \Z^n$ with $|B_i| > 2K$ and suppose $Q_1(x), Q_2(x) \neq 0$
for all $x \in B$. If $R_1 (x) = R_2 (x)$ for all $x \in B$ then $R_1 = R_2$ in $\C(x_1, \ldots, x_n)$.
\label{rat}
\end{cor}

\begin{thm}
Let $M$ be a cuspidal $AW_\mu$-module, $M = A \otimes U$, where $U = M_\beta$. Then the action of $W_\mu$ on
$M$ is given by
$$ (t^s d_\mu) (t^m u) = t^{m+s} (\mu\cdot m + D(s)) u, \quad u \in U,$$
where the family of operators $D(s): \, U \rightarrow U$ can be expressed as an $\End (U)$-valued polynomial in
$s = (s_1, \ldots, s_n)$ with the constant term $D(0) = \beta \Id$.
\label{poly}
\end{thm} 
\begin{proof}
We are going to prove this theorem by induction on $n$. In case when $n=1$, the solenoidal Lie algebra $W_\mu$ is isomorphic
to the Witt algebra $W_1$, and the claim follows from Theorem 1 in \cite{Jet}. 
Applying this result to a subalgebra spanned by $\left\{ t^{jm} d_\mu \, | j \in \Z \right\}$ with a fixed 
$m\in\Z^n \backslash \{ 0 \}$, which is isomorphic to $W_1$, we conclude that the family $\{ D(x m) \, | \, x \in \Z \}$ 
may be expressed as a polynomial in $x$. 

Now let us establish the step of induction. By induction assumption the operators $D(x_1, \ldots, x_{n-1}, 0)$ have
polynomial dependence on $x_1, \ldots, x_{n-1}$, while $D(0, \ldots, 0, x_n)$ is a polynomial in $x_n$. Thus
\begin{align*}
&[D(x_1, \ldots, x_{n-1}, 0),  D(0, \ldots, 0, x_n)] + \mu_n x_n D(0, \ldots, 0, x_n) \\
& \quad - (\mu_1 x_1 + \ldots \mu_{n-1} x_{n-1}) D(x_1, \ldots, x_{n-1}, 0) \\
&=  (\mu_n x_n - \mu_{n-1} x_{n-1} - \ldots - \mu_1 x_1) D(x_1, \ldots, x_{n-1}, x_n)
\end{align*}
is a polynomial in $x_1, \ldots, x_n$.

Since $\mu$ is generic then linear function on $\Z^n$
$$L_1(x) = \mu_n x_n - \mu_{n-1} x_{n-1} - \ldots - \mu_1 x_1$$ 
vanishes  only at $x = 0$. Thus for all $x \in \Z^n \backslash \{ 0 \}$, $ D(x_1, \ldots,  x_n)$ can be expressed as a rational
function $\frac{P_1 (x)}{L_1 (x)}$ with $P_1 \in \End (U) \otimes \C[x_1, \ldots, x_n]$.

Likewise,
\begin{align*}
& [D(x_1, \ldots, x_{n-1}-x_n, 0),  D(0, \ldots, x_n, x_n)] 
+ (\mu_{n-1} + \mu_n) x_n D(0, \ldots, x_n, x_n) \\
& \quad - (\mu_1 x_1 + \ldots + \mu_{n-1} x_{n-1}- \mu_{n-1} x_{n-1}) D(x_1, \ldots, x_{n-1}-x_n, 0) \\
&=  (\mu_n x_n + 2 \mu_{n-1} x_{n} - \mu_{n-1} x_{n-1} - \ldots - \mu_1 x_1) D(x_1, \ldots, x_{n-1}, x_n)
\end{align*}
is a polynomial in $x_1, \ldots, x_n$ and on the set $ \Z^n \backslash \{ 0 \}$, $ D(x_1, \ldots,  x_n)$ may be expressed as a rational
function $\frac{P_2 (x)}{L_2 (x)}$, where
$$L_2 (x) = \mu_n x_n  + 2 \mu_{n-1} x_{n} - \mu_{n-1} x_{n-1} - \ldots - \mu_1 x_1 .$$

By Corollary \ref{rat} we have
$\frac{P_1 (x)}{L_1 (x)} = \frac{P_2 (x)}{L_2 (x)}$
in  $\End (U) \otimes \C(x_1, \ldots, x_n)$, and hence $P_1(x) L_2(x) = P_2(x) L_1(x)$.
Since $\C[x_1, \ldots, x_n]$ is a unique factorization domain and the greatest common divisor of $L_1(x)$ and $L_2 (x)$
is $1$, we conclude that $P_1(x)$ is divisible by $L_1(x)$. Hence $D(x_1, \ldots, x_n)$ coincides with some 
polynomial on the set  $\Z^n \backslash \{ 0 \}$. It only remains to show that the value of this polynomial at $x=0$
coincides with $D(0)$. However the constant term of this polynomial is the same as the constant term 
of the polynomial $D(x_1, 0, \ldots, 0)$, and the constant term of the latter polynomial is $D(0)$ by the result for $W_1$.
This completes the proof of the theorem.
\end{proof}

 Expand the polynomial $D(s)$ in $s$:
$$D(s) = \sum_{k \in \Z_+^n} \frac{s^k}{k!} \del^k D, $$
where $\del^k D \in \End (U)$ with only a finite number of these operators being non-zero.
Here $k! = k_1 ! \times \ldots \times k_n !$.
Expanding the commutator $[D(s), D(m)]$ in $s$, $m$,
\begin{align*}
&\sum_{k, r \in \Z_+^n} \frac{s^k}{k!} \frac{m^r}{r!} [\del^k D, \del^r D] \\
= \left( \sum_{i=1}^n \mu_i m_i \right) 
&\sum_{p\in\Z_+^n} \frac{ (m+s)^p - m^p}{p!} \del^p D
-  \left( \sum_{i=1}^n \mu_i s_i \right) 
\sum_{p\in\Z_+^n} \frac{ (m+s)^p - s^p}{p!} \del^p D,
\end{align*}
we get the Lie brackets between $\del^k D$ by equating the coefficients at $\frac{s^k}{k!} \frac{m^r}{r!}$:
$$ [\del^k D, \del^r D] = \left\{
\begin{matrix}
&\sum\limits_{i=1}^n \mu_i (r_i - k_i) \del^{k+r-\epsilon_i} D &\text{ if \ } k,r \neq 0, \\
&0 \hspace{3cm} &\text{ if \ } k=0 \text{\rm \ or \ } r = 0. \hfill \\
\end{matrix}
\right.$$
Here $\epsilon_i$ is the $i$-th standard basis vector of $\Z^n$. We can identify 
$\Span \left\{ \del^k D \, | \, k \in \Z_+^n \backslash \{ 0 \} \right\} $
with a subalgebra in $\Der \C [x_1, \ldots, x_n]$, where
$\del^k D \mapsto x^k \del_\mu$, $\del_\mu = \mu_1 \frac{\del}{\del x_1} + \ldots + \mu_n \frac{\del}{\del x_n}$.
Set $\LL = A \del_\mu \subset \Der \C [x_1, \ldots, x_n]$. The Lie algebra $\LL$ has a $\Z$-grading defined by 
assigning $\deg (x_i) = 1$, $i = 1, \ldots, n$, $\deg \del_\mu = -1$:
$$ \LL = \mathop\oplus\limits_{j=-1}^\infty \LL_j .$$
Then $\Span \left\{ \del^k D \, | \, k \in\Z_+^n \backslash \{ 0 \} \right\}$ is identified with the subalgebra
$$\LL_+ = \mathop\oplus\limits_{j=0}^\infty \LL_j .$$
We obtained the following result
\begin{thm}
(cf., \cite{Jet}, Theorem 3)
There exists an isomorphism between the category of cuspidal $AW_\mu$-modules with support $\beta+\Gamma_\mu$
and the category of finite-dimensional $\LL_+$-modules $(U,\rho)$ satisfying 
\begin{align}
\rho(x^k \del_\mu) = 0 \quad\text{ \ for \ } k_1+\ldots+k_n \gg 0.
\label{locn}
\end{align}
Given such an $\LL_+$-module $U$, we associate to it an $AW_\mu$-module
$$M = A \otimes U$$
with the following action of $W_\mu$:
\begin{equation}
(t^s d_\mu) (t^m \otimes u) = t^{m+s} \otimes \left( (\mu\cdot m + \beta) \Id + \sum_{k\in\Z_+^n \backslash \{ 0 \}}
\frac {s^k}{k!} \rho(x^k \del_\mu) \right) u .
\end{equation}
\label{AW}
\end{thm} 
\begin{rem}
We shall see below that condition \eqref{locn} automatically holds in every finite-dimensional $\LL_+$-module $U$ and hence can be dropped 
from the statement of the theorem.
\end{rem}
\begin{rem}
The Lie algebra $\LL$ is the jet Lie algebra for the solenoidal Lie algebra $W_\mu$ (see \cite{BIM} for the definition of the jet Lie algebra).
\end{rem}

Now let us focus on the description of simple cuspidal $AW_\mu$-modules. According to Theorem \ref{AW}, such modules correspond to 
simple finite-dimensional $\LL_+$-modules satisfying \eqref{locn}.
\begin{thm}
(1) The commutant $\LL_+^{(1)}$ has codimension 1 in $\LL_+$:
$$ \LL_+^{(1)} = \LL_0^{(1)} \oplus \sum_{j=1}^\infty \LL_j ,$$
where the commutant $\LL_0^{(1)}$ of $\LL_0$ is an abelian subalgebra
$$\LL_0^{(1)} = \left\{ \sum_{i=1}^n c_i x_i \del_\mu \, \bigg| \, \sum_{i=1}^n \mu_i c_i = 0 \right\} .$$

(2) Every finite-dimensional module $U$ for $\LL_+$ satisfies the condition \eqref{locn}.

(3) Every finite-dimensional simple $\LL_+$-module has dimension 1.

(4) All 1-dimensional representations of $\LL_+$ are parametrized by $\alpha\in\C$ with the action
of $\LL_+$ given by $x_i \del_\mu \mapsto \alpha \mu_i$, $i = 1, \ldots, n$, 
$\LL_j \mapsto 0$ for $j \geq 1$.
\label{irL}
\end{thm}
\begin{cor}
Every simple cuspidal $AW_\mu$-module is isomorphic to a tensor module $T(\alpha, \beta)$ for some $\alpha, \beta \in \C$.
\label{simpleAW}
\end{cor}
\begin{proof}
Combining Theorems \ref{AW} and \ref{irL}, we see that a simple cuspidal $AW_\mu$-module $M$ is isomorphic to
$A \otimes U$, where $U = \C u$ with the action 
$$(t^s d_\mu) (t^m \otimes u) = (\mu\cdot m + \beta + \alpha \mu \cdot s) t^{m+s} \otimes u .$$
This module is isomorphic to $T(\alpha, \beta)$.
\end{proof}
{\it Proof of Theorem \ref{irL}.}
First of all, let us compute the commutant  of $\LL_0$. It is easy to see that the commutators
$$ [x_i \del_\mu, x_j \del_\mu ] = \mu_j x_i \del_\mu - \mu_i x_j \del_\mu$$
span the subspace
$$\left\{ \sum_{i=1}^n c_i x_i \del_\mu \; \bigg| \; \sum_{i=1}^n \mu_i c_i = 0 \right\} .$$
Next, let us show that for every $p \geq 1$, $[ \LL_0, \LL_p ] = \LL_p$. Note that
\begin{equation}
[x_j \del_\mu, x^k \del_\mu ] = \mu_j (k_j -1) x^k \del_\mu + f(x) \del_\mu ,
\label{lead}
\end{equation}
where all monomials in $f$ have degree $k_j + 1$ in $x_j$. Using a descending induction on $k_j$
we conclude that $x^k \del_\mu \in \LL_+^{(1)}$ if $k_j \geq 2$. It remains to show that
$x_{i_1} x_{i_2} \ldots x_{i_{p+1}} \del_\mu \in \LL_+^{(1)}$, $i_1 < \ldots < i_{p+1}$.

We have
$$[x_{i_1} \del_\mu, x_{i_2}^2 x_{i_3} \ldots x_{i_{p+1}} \del_\mu] =
2 \mu_{i_2} x_{i_1} x_{i_2} \ldots x_{i_{p+1}} \del_\mu + g(x) \del_mu ,$$
where $g(x)$ is divisible by $x_{i_2}^2$. Since $g(x) \del_\mu \in \LL_+^{(1)}$,
we get that $x_{i_1} \ldots x_{i_{p+1}} \del_\mu \in \LL_+^{(1)}$ and part (1) is proved.

Let us prove part (2). Consider the action of $\ad (x_j \del_\mu)$ on $\LL_p$. It follows from \eqref{lead}
that $x^k \del_\mu$ belongs to the subspace spanned by the eigenvectors for  $\ad (x_j \del_\mu)$
with eigenvalues $\mu_j s$, where $s \geq k_j -1$. By Lemma 2 from \cite{Jet}, given
a finite-dimensional $\LL_+$-module $(U, \rho)$, there exists $m \in \N$ such that for all 
$j = 1, \ldots, n$, every eigenvector of  $\ad (x_j \del_\mu)$ with an eigenvalue $\mu_j s$, with $s \geq m$,
acts trivially on $U$. Thus $\rho(x^k \del_\mu) = 0$ if $k_j -1 \geq m$ for some $j$. This implies
$\rho(\LL_p) = 0$ if $p \geq nm$, which establishes claim (2).

 It follows from (2) that every finite-dimensional representation of $\LL_+$ factors through the quotient
$\LL_+ / \sum\limits_{j \geq p} \LL_j$ for some $p$. This quotient is a finite-dimensional solvable
Lie algebra and by the Lie's Theorem (\cite{Bour}, Corollary I.5.3.2), its finite-dimensional irreducible representations 
have dimension one.

Finally, to prove part (4), we note that 1-dimensional representations of $\LL_+$ are determined by 
1-dimensional representations of $\LL^+ / \LL_+^{(1)}$. Since 
$\dim \left( \LL^+ / \LL_+^{(1)} \right) = 1$, such representations are described by a single parameter
$\alpha \in \C$. The map
$$\varphi : \quad \LL_+ / \LL_+^{(1)} \rightarrow \C$$
is given by $\varphi( x_j \del_\mu) = \mu_j$, $1\leq j \leq n$, $\varphi(\LL_p) = 0$ for $p \geq 1$.
Claim (4) now follows. 
\hfill\qed 
\section{$A$-cover of a cuspidal $W_\mu$-module}
In this section we shall establish a relation between cuspidal $W_\mu$-modules and cuspidal
$AW_\mu$-modules. We are going to prove the following theorem:
\begin{thm}
Let $M$ be a cuspidal $W_\mu$-module satisfying $W_\mu M = M$. Then there exists a cuspidal
$AW_\mu$-module $\widehat{M}$ and a surjective homomorphism of $W_\mu$-modules
$\widehat{M} \rightarrow M$.
\label{cover}
\end{thm} 
Following \cite{BF}, we define certain elements in the universal enveloping algebra of $W_\mu$.
\begin{dfn}
For $k, s, h \in \Gamma_\mu$ we define a {\it differentiator} of order $m$ as
$$\Omega_{k,s}^{(m,h)} = \sum_{i=0}^m (-1)^i {m \choose i} e_{k-ih} e_{s+ih} \in U(W_\mu). $$
\end{dfn}
The following theorem shows that a subspace in $U(W_\mu)$ spanned by products of differentiators 
of a given order, contains differentiators of higher orders. Here we denote by $\{ X, Y \}$ the
anticommutator $XY + YX$.
\begin{thm} (\cite{BF}) 
Let $r \geq 2$, $k,s,p,q,h \in \Gamma_\mu$. Then
\begin{align}
\sum_{i=0}^r \sum_{j=0}^r (-1)^{i+j} {r \choose i} {r \choose j}
& \left( \left\{ \Omega^{(r,h)}_{k-ih,s-jh} , \Omega^{(r,h)}_{q+ih,p+jh} \right\}
-       \left\{ \Omega^{(r,h)}_{k-ih,q-jh} , \Omega^{(r,h)}_{s+ih,p+jh} \right\} \right) \notag \\
& = (q-s) (p-k+2rh) \Omega^{(4r,h)}_{k+p+2rh,s+q-2rh} .
\label{Om}
\end{align}
\end{thm}
Using this theorem we are going to prove that every cuspidal $W_\mu$-module is annihilated by 
the differentiators of high enough orders. For the Witt algebra $W_1$ we can quote a result
from \cite{BF}:
\begin{thm}
(\cite{BF})
For every $\ell \in \N$ there exists $r \in \N$ such that for all $k, s \in \Z$ the differentiators
$$\Omega^{(r)}_{k,s} =   \sum_{i=0}^r (-1)^i {r \choose i} e_{k-i} e_{s+i}$$
annihilate all cuspidal $W_1$-modules with a composition series of length $\ell$.
\end{thm}
Note that this result may be stated in terms of the bound on the dimensions of weight spaces
instead of the length of the composition series, since for a cuspidal $W_1$-module $M$ the length
of its composition series does not exceed $\dim M_0 + \dim M_\lambda$, where $\lambda \in \beta + \Gamma_\mu$,
$\lambda \neq 0$.

 For any non-zero $h \in\Gamma_\mu$, we can view a $W_\mu$-module as a module for the subalgebra 
with basis $\{ e_{ih} \, | \, i \in\Z\}$, which is isomorphic to $W_1$. Then a cuspidal $W_\mu$-module
becomes a direct sum of cuspidal $W_1$-submodules, corresponding to the cosets of $\Z h$ in $\Gamma_\mu$,
and all summands have a common bound on the dimensions of weight spaces. As a consequence we get
\begin{cor}
Let $M$ be a cuspidal $W_\mu$-module. Then there exists $r \in \N$, which depends only on the bound
for the dimensions of weight spaces of $M$, such that $\Omega^{(r,h)}_{q,p}$ annihilates $M$ for all
$h \in \Gamma_\mu$, $q, p \in \Z h \subset \Gamma_\mu$.
\label{Wone}
\end{cor}
Now let us establish a generalization of the above corollary:
\begin{prp}
Let $M$ be a cuspidal $W_\mu$-module. Then there exists $m \in \N$, which depends only on the bound
on the dimensions of weight spaces of $M$, such that $\Omega^{(m, h)}_{k,s}$ annihilates $M$
for all $h, k \in \Gamma_\mu$ and $s \in \Z h$.
\label{half}
\end{prp}
\begin{proof}
Fix $h \in \Gamma_\mu$. By Corollary \ref{Wone} there exists $r\in \N$ such that for all $q, p \in \Z h$,
$\Omega^{(r,h)}_{q,p} \in \Ann (M)$. Then for $k \in \Gamma_\mu$, $s, q, p \in \Z h$, the left hand side
of \eqref{Om} annihilates $M$. Thus  $\Omega^{(4r,h)}_{k+p+2rh,s+q-2rh} \in \Ann (M)$.
Setting $m = 4r$, we get the claim of the proposition.
\end{proof}

 Now let us proceed with the proof of Theorem  \ref{cover}.

Given a cuspidal $W_\mu$-module $M$, we construct the coinduced module $\Hom (A, M)$, which
is an $AW_\mu$-module with the following action of $A$ and $W_\mu$ (\cite{BF}, Proposition 4.3):
\begin{align*}
&(x \phi) (f) = x (\phi(f)) - \phi(xf), \hspace{4cm} \\
&(g \phi ) (f) = \phi(gf),  \quad \phi\in\Hom (A, M), x \in W_\mu, f, g \in A.
\end{align*}
The map $\pi: \; \Hom(A, M) \rightarrow M$, $\pi (\phi) = \phi(1)$, is a surjective homomorphism 
of $W_\mu$-modules.
\begin{dfn} (\cite{BF})
An $A$-cover of a cuspidal $W_\mu$-module $M$ is an $AW_\mu$-submodule $\widehat{M}$ in the coinduced
module $\Hom (A, M)$, spanned by
$$\left\{ \psi(x, u) \, | \, x \in W_\mu, u \in M \right\}, $$
where
$\psi (x, u) : A \rightarrow M$ is given by
$$\psi(x, u) (f) = (fx) u .$$
\end{dfn}
The $AW_\mu$-action on $\widehat{M}$ is
\begin{align*}
&y \psi(x, u) = \psi( [y,x], u) + \psi(x, yu), \hspace{3cm} \\
&g \psi(x, u) = \psi(gx, u), \quad x, y \in W_\mu, u \in M, g \in A .
\end{align*}
The map $\pi: \; \widehat{M} \rightarrow M$, $\pi(\psi(x, u)) = \psi(x, u) (1) = x u$, is a homomorphism of 
$W_\mu$-modules with $\pi(\widehat{M}) = W_\mu M$.

It turns out that $\widehat{M}$ is a cuspidal $AW_\mu$-module. The proof of this fact is the same as given in \cite{BF}, and its key
ingredient is Proposition \ref{half}. This establishes Theorem \ref{cover}.
\hfill\qed

Now we can prove the main result of the paper.

\begin{thm}
Let $\mu\in\C$ be generic. Every simple cuspidal $W_\mu$-module is isomorphic to either

$\bullet$ a module of tensor fields $T(\alpha, \beta)$ with $\alpha \neq 0, 1$ or $\alpha = 0, 1$ and
$\beta \not\in\Gamma_\mu$,

$\bullet$ the quotient $\overline{T}(0,0)$ of the module of functions by the 1-dimensional submodule
$\C v_0$ of constant functions,

or

$\bullet$ the trivial 1-dimensional module.
\label{main}
\end{thm}
\begin{proof}
Let $M$ be a simple cuspidal $W_\mu$-module. Let us assume that it is different from the trivial 1-dimensional module.
Then $W_\mu M = M$. Construct the $A$-cover $\widehat{M}$ of $M$ with a surjective homomorphism of $W_\mu$-modules
$\pi: \; \widehat{M} \rightarrow M$. Consider the composition series
$$ 0 = \widehat{M}_0 \subset  \widehat{M}_1 \subset \ldots \subset  \widehat{M}_\ell =  \widehat{M}$$
with the quotients $ \widehat{M}_{i+1} /  \widehat{M}_i$ being simple $AW_\mu$-modules. Let $k$
be the smallest integer such that $\pi ( \widehat{M}_k ) \neq 0$. Since $M$ is a simple $W_\mu$-module,
we have $\pi ( \widehat{M}_k) = M$ and  $\pi ( \widehat{M}_{k-1}) = 0$.
The map $\pi$ factors through to
$$\overline{\pi}: \;  \widehat{M}_k /  \widehat{M}_{k-1} \rightarrow M .$$
The module $\widehat{M}_k /  \widehat{M}_{k-1} $ is a simple cuspidal $AW_\mu$-module, and by Corollary \ref{simpleAW}
is isomorphic to a module of tensor fields $T(\alpha, \beta)$ for some $\alpha, \beta \in\C$. Thus $M$ is a homomorphic
image of $T(\alpha, \beta)$.

By Proposition \ref{red}, $T(\alpha, \beta)$ is simple as a $W_\mu$-module if  $\alpha \neq 0, 1$ or $\alpha = 0, 1$ and
$\beta \not\in\Gamma_\mu$, thus $M \cong T(\alpha, \beta)$ in this case. If $\alpha = 0$ and $\beta \in \Gamma_\mu$
then the unique simple quotient of $T(0,0)$ is $\overline{T}(0,0)$. Finally, in case when $\alpha =1$ and $\beta\in\Gamma_\mu$,
the unique simple quotient of $T(1,0)$ is the trivial 1-dimensional module. This shows that the list given in the statement of the
theorem is complete.
\end{proof}

\

\section{Acknowledgements}
The first author is supported in part by a grant from the Natural Sciences
and Engineering Research Council of Canada.  
The second author is supported in part by the CNPq grant
(301743/2007-0) and by the Fapesp grant (2010/50347-9).
Part of this work was carried out during the visit of the first author
to the University of S\~ao Paulo in 2013. This author would like
to thank the University of S\~ao Paulo for hospitality and
excellent working conditions and Fapesp (2012/14961-0) for
financial support.

\end{document}